\documentclass[reqno]{amsart}

\usepackage{amsmath}
\usepackage{amssymb}
\usepackage{amscd}
\usepackage{tikz-cd}
\usepackage{tikz}

\usetikzlibrary{positioning}

\begin{document}

\title{The Natural partial order on modules}

\author{T. Pakel}
\address{Tugba Pakel,  Department of Mathematics, Ankara University, Turkey}
\email{tpakel@ankara.edu.tr}

\author{T. P. Calci}
\address{Tugce Pekacar Calci,  Department of Mathematics, Ankara University, Turkey}
\email{tcalci@ankara.edu.tr}

\author{S. Halicioglu}
\address{Sait Hal\i c\i oglu,  Department of Mathematics, Ankara University, Turkey}
\email{halici@ankara.edu.tr}

\author{A. Harmanci}
\address{Abdullah Harmanci, Department of Mathematics, Hacettepe University,  Turkey}
\email{harmanci@hacettepe.edu.tr}

\author{B. Ungor}
\address{Burcu Ungor,  Department of Mathematics, Ankara University, Turkey}
\email{bungor@science.ankara.edu.tr}

\date{}
\newtheorem{thm}{Theorem}[section]
\newtheorem{lem}[thm]{Lemma}
\newtheorem{prop}[thm]{Proposition}
\newtheorem{cor}[thm]{Corollary}
\newtheorem{exs}[thm]{Examples}
\newtheorem{defn}[thm]{Definition}
\newtheorem{nota}{Notation}
\newtheorem{rem}[thm]{Remark}
\newtheorem{ex}[thm]{Example}
\newtheorem{que}[thm]{Question}

\begin{abstract}
The Mitsch order is already known as a natural partial order for semigroups and rings. The purpose of this paper is to further  study of the Mitsch order on modules by investigating basic properties via  endomorphism rings. And so this study also contribute to the results related to the orders on rings. As a module theoretic analog of the Mitsch order, we show that this order is a partial order on arbitrary modules. Among others, lattice properties of the Mitsch order and the relations between the Mitsch
order and the other well-known orders, such as, the minus order,
the Jones order, the direct sum order and the space pre-order on
modules are studied. In particular, we prove that the minus order is the Mitsch order and we supply an example to show that the converse does not hold in general.

 \vspace{2mm}


\end{abstract}

\maketitle

\section{Introduction}
Throughout this paper $R$ denotes an associative ring with
identity $1_R$ and modules are unitary right $R$-modules. For a
module $M$, $S =$ End$_R(M)$ stands for the ring of all
right $R$-module endomorphisms of $M$. It is well known that $M$
is a left $S$-right $R$-bimodule. In this work, for the $(S,
R)$-bimodule $M$, $l_S(.)$ and $r_R(.)$ stand for the left
annihilator of a subset of $M$ in $S$ and the right annihilator of
a subset of $M$ in $R$, respectively. If the subset is a
singleton, say $\{m\}$, then we simply write $l_{S}(m)$ and
$r_{R}(m)$, respectively. For a ring $R$, $C(R)$ denotes
the center of $R$.

Let $\mathcal{S}$ be a semigroup and $a\in\mathcal{S}$. Any
solution $x=a^{-}$ to the equation $axa=a$ is called an \textit{inner generalized inverse} of $a$. If in addition $xax=x$, then
$x$ is called a \textit{reflexive inverse} of $a$ and denoted by
$a^+$. If such $a^{-}$ exists, then $a$ is called
\textit{regular}, and if every element in a semigroup
$\mathcal{S}$ is regular, then $\mathcal{S}$ is called a \textit{
regular semigroup}. Hartwig \cite{Hartwig} introduced the \textit{plus order} $\leq^{+}$ on regular semigroups using generalized
inverses. For a regular semigroup $\mathcal{S}$ and
$a,b\in\mathcal{S}$, we write
\begin{equation}
a\leq^{+}b\quad\text{if\quad}a^{+}a=a^{+}b\quad\text{and\quad}aa^{+}=
ba^{+}
\label{HartwigPlusDef}%
\end{equation}
for some reflexive inverse $a^{+}$ of $a$. It is shown that the
plus order is a partial order. On the other hand, consider
\begin{equation}
a\leq^{-}b\quad\text{if\quad}a^{-}a=a^{-}b\quad\text{and\quad}aa^{-}=ba^{-}
\label{HartwigMinusDef}%
\end{equation} for some inner generalized inverse $a^{-}$ of $a$.
The relation $\leq^{-}$ is usually called \textit{minus order}. It
is well known that definitions (\ref{HartwigPlusDef}) and
(\ref{HartwigMinusDef}) are equivalent. In \cite{HS}, the plus
partial order was renamed as minus partial order.

Let $S$ be any semigroup and $S^1$ denote the
set $S$ if  $S$ has an identity.  In \cite{Mitsch}, the Mitsch order was introduced for semigroups by Mitsch as follows: for $a,b\in S$, \begin{center}
$a \leq b$ if $a = xb = by$, $xa = a$ \end{center} for some $x, y \in S^1$. Then it is a partial order on $S$ which is called the
\textit{Mitsch order}.

Principally projective rings
were introduced by Hattori \cite{H} to study the torsion theory, that is, a ring is called
{\it left (resp. right) principally projective} if every principal left (resp. right) ideal is projective. This
is equivalent to the left (resp. right) annihilator of any element of the ring is generated
by an idempotent as a left (resp. right) ideal, i.e., the ring is {\it left (right) Rickart}. If a ring $R$ is both left and right Rickart, then it is said to be {\it Rickart}. 

Some partial orders are defined on Baer and Rickart  rings in the literature (see \cite{DjordjevicRakicMarovt}, \cite{M} and \cite{UHHM}). Let $R$ be a ring  and $a,b\in R$. Then we write $a\leq^{-}b$ if there
exist idempotent elements $p,q\in R$ such that
\begin{equation}
l_{R}(a)=R(1_R-p),  \quad r_{R}(a)=(1_R-q)R, \quad pa=pb, \quad
\text{and} \quad aq=bq. \label{def_minus_Rickart}
\end{equation}
It was proved in \cite{DjordjevicRakicMarovt} that this relation
$\leq^{-}$ is indeed a partial order when $R$ is a Rickart ring. In \cite{Drazin}, the Mitsch order defined in a ring in the following way: For any ring $R$, a natural partial order $\mathcal M$ on $R$ defined by $x \mathcal M y$ for given $x,y \in R$ if either $x=y$ or there exist $p,q \in R$ such that $px=py=x=xq=yq$. In \cite{DKM}, it is shown that  the minus order and the Mitsch order are the same on Rickart rings.

Let $M$ be a module. Then an element $m \in M$ is called
({\it Zelmanowitz}) {\it regular} if $m=m\varphi (m) \equiv \hspace{0,1cm} m\varphi m$ for
some $\varphi\in M^*$ where $M^* =$ Hom$_R(M, R)$ denotes the dual
of $M$. For a ring $R$, let $a\in R$ be a regular element (in the sense of von Neumann), i.e., there exists $a^- \in R$ such that $a=aa^-a$. Define the map $\varphi:R\rightarrow R$ with
$\varphi (r)=a^-r, r\in R$. Then $\varphi \in R^* =$ End$_R(R)$. We have $a\varphi a=a\varphi(a)=aa^-a=a$ which yields that $a$ is regular in $R_R$ (in the
sense of Zelmanowitz, see also \cite{ZEL}). Let now $R$ be a ring and suppose $a$ is a regular element in $R_R$ (in the sense of Zelmanowitz), i.e., there exists  $\varphi\in R^* $ such that $a=a\varphi a$. Define $a^-=\varphi(1_R)$. Then  $a=a\varphi a=a\varphi(1_R)a=a\varphi(1_Ra)=a\varphi(a)=a\varphi a=a.$
We may conclude that $a \in R$ is regular if and only if $a$ is regular in
$R_R$ (or, similarly, in the left-$R$ module $_RR$). A module $M$ is called \textit{regular} (in the
sense of Zelmanowitz) if every element of $M$ is
regular.

Recently, some of the present authors study the minus order in a more general setting. The notion of the minus order was extended to modules using their endomorphism rings in \cite{UHHM-comm}. Let $M$ be a module and $m_1$, $m_2 \in M$. Then $m_1 \leq^- m_2$ if there exists $\varphi \in M^*$ such that $m_1=m_1\varphi m_1$, $m_1\varphi=m_2 \varphi$ and $\varphi m_1= \varphi m_2$.  They called the relation $\leq^-$ the {\it minus order} on $M$, and proved that it is a partial order on regular modules. In the same work, the authors also defined the Mitsch order on modules, that is, $m_1\leq_M m_2$ if there exist $f \in S$ and
$a \in R$ such that $m_1= m_2a= fm_2$ and $m_1= fm_1$. It is showed in \cite{UHHM-comm} that the minus order and the Mitsch order coincide on regular modules. 

Although the minus order on various algebraic structures has been an area of intense research in the past few years, the main properties of the Mitsch order is not studied in detail. Motivated by the aforementioned works on the minus  order, the goal of this paper is to investigate properties of the Mitsch order on arbitrary modules. The paper is organized as follows: In Section 2, we show that the
Mitsch relation is a partial order on modules, and give some characterizations of this partial order in various ways.
We determine the way how to define the Mitsch order from a module to its factor modules. In Section 3, the relations between the Mitsch
order and the other well-known orders on modules are studied and the following diagram is obtained:\\
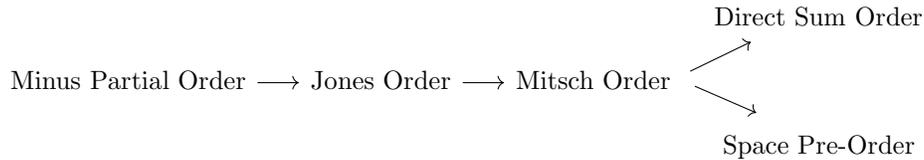
\begin{tikzcd}[column sep=0.8em,row sep=1em,style = {font =\normalsize}]
  &&&&&\makebox{Direct Sum Order} \arrow[dl,leftarrow, shorten=0.2cm,end anchor={[xshift=-0.5ex]},end anchor={[yshift=-1.5ex]north east}]  \\
  \makebox{Minus Partial Order} \arrow[rr,rightarrow] && \makebox{Jones Order} \arrow[rr,rightarrow]  && \makebox{Mitsch Order}\\
  &&&&&\makebox{Space Pre-Order}\arrow[ul,leftarrow,shorten <=0.4cm,shorten >=0.28cm,end anchor={[xshift=-0.6ex]},end anchor={[yshift=-1.5ex]north east}]
\end{tikzcd}\\
We supply examples to show that all implications in this diagram are strict. Section 4 deals with lattice properties of the Mitsch order on modules. The maximal elements of modules with respect to the Mitsch order are determined. Finally, we focus on the compatibility of the  Mitsch order  with the addition and multiplication.

\section{Basic Properties of the Mitsch Order}

In \cite{Mitsch}, Mitsch generalized Hartwig's definition of the
minus partial order to arbitrary semigroups. Suppose $a,b$ are two elements of an arbitrary semigroup
$\mathcal{S}$. Then we write
\begin{equation}
a\leq_{M}b\text{\quad if\quad}a=xb=by \text{ and }xa=a \label{MitschDefMinus}%
\end{equation}
for some elements $x,y\in\mathcal{S}^{1}$, where $S^1$ denotes the
set $S$ if  $S$ has an identity. Mitsch proved that
$\leq_{M}$ is indeed a partial order for any semigroup
$\mathcal{S}$ and that for $a,b\in\mathcal{S}$,
\begin{equation} \label{MitschMinusEquiv}
a\leq_{M}b\text{\quad if and only if\quad}a=xb=by\text{ and
}xa=ay=a
\end{equation}
for some elements $x,y\in\mathcal{S}^{1}$. One can easily
generalize this definition to a ring $R$ as follows.
\begin{defn}
\label{Definition_Mitsch_on rings} \rm Let $R$ be a ring and $a,b \in
R$. We write $a\leq _{M}b$ if there exist $x,y \in R$ such that $a
=xb=by$ and $a=xa$.
\end{defn}

It can be easily seen that the Mitsch order is a partial order on
rings.

\begin{prop} \label{mitschdenklik} Let $M$ be a module and $m_1, m_2\in M$. Then the following statements are equivalent.
\begin{enumerate}
    \item There exist $f \in S$ and $a \in R$ such that $m_1 =m_2a=fm_2$ and $m_1=fm_1$.
    \item There exist $f \in S$ and $a \in R$ such that $m_1 =m_2a=fm_2$ and $m_1=m_1a$.
    \item There exist $f \in S$ and $a \in R$ such that $m_1 =m_2a=fm_2$ and $m_1=fm_1=m_1a$.
\end{enumerate}\end{prop}

\begin{proof}
(1) $\Rightarrow$ (2) If there exist $f \in S$ and $a \in R$ such that $m_1 =m_2a=fm_2$ and $m_1=fm_1$, then $m_1=fm_2a=m_1a$. \\
(2) $\Rightarrow$ (3) If there exist $f \in S$ and $a \in R$ such that $m_1 =m_2a=fm_2$ and $m_1=m_1a$, then $m_1=m_1a=fm_2a=fm_1$.\\
(3) $\Rightarrow$ (1) It is obvious.
\end{proof}

Following the above ideas, we now consider the relation
$\leq_{M}$ on  a module $M$ which was defined in \cite{UHHM-comm}. 

\begin{defn}
\label{Definition_Mitsch} \rm Let $M$ be a module and $m_1,m_2 \in M$.
We write $m_1\leq _{M}m_2$ if one of the equivalent conditions in
Proposition \ref{mitschdenklik} is satisfied. We call the relation $\leq_{M}$ the \it{Mitsch order} on $M$.
\end{defn}

In the sequel of this section, we examine  the main properties of the Mitsch order on modules in detail.

\begin{thm}\label{mitsch} Let $M$ be a module. The Mitsch order is a partial order
on $M$.
\end{thm}
\begin{proof} \textbf{Reflexivity}: Obvious.\\
\textbf{Antisymmetry:} Let $m_1, m_2\in M$ with $m_1\leq_{M} m_2$
and $m_2\leq_{M} m_1$. Then there exist $f,g \in
S$ and $a,b \in R$
such that
$$ m_1 =m_2a=fm_2,~~m_1=fm_1 $$ and
$$ m_2 =m_1b=gm_1,~~m_2=gm_2.$$
Hence $m_1=m_2a=gm_2a=gm_1=m_2$.\\
\textbf{Transitivity:} Let $m_1, m_2, m_3\in M$ with $m_1\leq_{M} m_2$
and $m_2\leq_{M} m_3$. Then there exist $f,g \in
S$ and $a,b \in R$
such that
$$ m_1 =m_2a=fm_2,~~m_1=fm_1 $$ and
$$ m_2 =m_3b=gm_3,~~m_2=gm_2.$$
Hence $m_1=m_3(ba)=(fg)m_3$ and $m_1=fm_1=fm_2a=fgm_2a=(fg)m_1$ and so $m_1\leq_{M} m_3$. Therefore the Mitsch order is a partial order on $M$.
\end{proof}

In the next result, we show that there are many different ways to
express the Mitsch order.

\begin{thm} \label{nlidenklikler}
Let $M$ be a module and $m_1,m_2 \in M$. The following are
equivalent.
\begin{enumerate}
    \item $m_1\leq_{M} m_2$.
    \item There exist $f \in
S$ and $a \in R$ such that $m_1 =m_2a^n=f^nm_2$ and $m_1=f^nm_1$,
for any $n\in \mathbb{Z}^+$.
    \item There exist $f \in
S$ and $a \in R$ such that $m_1 =m_2a^n=f^nm_2$ and $m_1=m_1a^n$,
for any $n\in \mathbb{Z}^+$.
    \item There exist $f \in
S$ and $a \in R$ such that $m_1 =m_2a^n=f^nm_2$ and
$m_1=f^nm_1=m_1a^n$, for any $n\in \mathbb{Z}^+$.
\end{enumerate}
\end{thm}

\begin{proof}
(1) $\Leftrightarrow$ (2) Let $m_1\leq_{M} m_2$, then there exist $f \in S$ and $a \in R$ such that $m_1=m_2a=fm_2$ and $m_1=fm_1$. We show that $m_1 =m_2a^n=f^nm_2$ and $m_1=f^nm_1$ for any $n\in \mathbb{Z}^+$. If $n=1$ there is nothing to show. Assume $n>1$ then $m_1=fm_1=ffm_1=\cdots =\underbrace{ff\cdots f\,}_\text{$n$-times}m_1=f^nm_1$, $m_1=fm_1=ffm_2=\cdots =\underbrace{ff\cdots f\,}_\text{$n$-times}m_2=f^nm_2$  and $m_1=m_2a=m_2aa=\cdots=m_2\underbrace{aa\cdots a\,}_\text{$n$-times}=m_2a^n$. The converse is clear.\\
(1) $\Leftrightarrow$ (3) and (1) $\Leftrightarrow$ (4) can be proved similar to the proof of (1) $\Leftrightarrow$ (2).
\end{proof}

The following corollary is a direct consequence of Theorem
\ref{nlidenklikler} which is a new characterization of the Mitsch
order on rings.

\begin{cor}
Let $R$ be a ring and $a, b \in R$. The following are
equivalent.
\begin{enumerate}
    \item $a\leq_{M} b$.
    \item There exist $x,y \in R$ such that $a=bx^n=y^nb$ and $a=y^na$,
for any $n\in \mathbb{Z}^+$.
    \item There exist $x,y \in R$ such that $a=bx^n=y^nb$ and $a=ax^n$,
for any $n\in \mathbb{Z}^+$.
    \item There exist $x,y \in R$ such that $a=bx^n=y^nb$ and
$a=y^na=ax^n$, for any $n\in \mathbb{Z}^+$.
\end{enumerate}
\end{cor}

In the next result, we give some basic properties of the Mitsch order.

\begin{prop} \label{mitschann}
Let $M$ be a module and $m_1, m_2\in M$. If $m_1\leq_{M} m_2$, then
the following hold.
\begin{enumerate}
    \item $l_{S}(m_2)\subseteq l_{S}(m_1)$.
    \item $r_{R}(m_2)\subseteq r_{R}(m_1)$.
    \item $m_1R\subseteq m_2R$.
    \item $Sm_1\subseteq Sm_2$.
\end{enumerate}
\end{prop}

\begin{proof}
It is straightforward.
\end{proof}

We now give an example to show that the converse of Proposition
\ref{mitschann} need not be true.

\begin{ex} \rm
Let $R=M_2(\mathbb Z)$, and $M=R_R$. Consider
       $m_1=\begin{bmatrix}
       3 & 0 \\
       0 & 3
       \end{bmatrix}$,
       $m_2=\begin{bmatrix}
       -1 ~& 2 \\
        2 &-1
       \end{bmatrix}$ and
       $a=f=\begin{bmatrix}
       1 & 2 \\
       2 & 1
       \end{bmatrix}$
       $\in M$. Since $m_1=m_2a=fm_2$ we have $l_{S}(m_2)\subseteq l_{S}(m_1)$,
       $r_{R}(m_2)\subseteq r_{R}(m_1)$, $m_1R\subseteq m_2R$, $Sm_1\subseteq
       Sm_2$.
       On the other hand, it can be seen that $m_1\nleq_{M}
       m_2$.
\end{ex}

 Let $M$ be a module and $N$ a submodule of $M$. It is
natural to ask if it is possible to extend the Mitsch order from
$N$ to $M$ and vice versa. The following result answers these
questions.

\begin{prop}
 Let $M$ be a module, $N$ a submodule of $M$ and $m_1, m_2\in N$.
 \begin{enumerate}
     \item If $N$ is a direct summand of $M$ and $m_1\leq_{M} m_2$ on $N$, then $m_1\leq_{M} m_2$ on $M$.
     \item  If $N$ is fully invariant in $M$ and $m_1\leq_{M} m_2$ on $M$, then $m_1\leq_{M} m_2$ on $N$.
 \end{enumerate}
\end{prop}

\begin{proof} (1)
Let $m_1\leq_{M} m_2$ on $N$. Then there exist $f \in$ End$_R(N)$
and $a \in R$ such that $m_1=m_2a=fm_2$ and $m_1=fm_1$. If $N$ is
a direct summand of $M$, then there exists a submodule $K$ of $M$
such that $M=N\oplus K$. Hence $f\oplus0_K\in S$ and
$m_1=m_2a=(f\oplus0_K)m_2$ and $m_1=(f\oplus0_K)m_1$. Therefore
$m_1\leq_{M} m_2$ on $M$.\\
(2) Since $m_1\leq_{M} m_2$ on $M$, there exist $f\in S$ and $a\in R$ such that $m_1=fm_2=m_2a$ and $m_1=fm_1$. The submodule $N$ being fully invariant yields that $f{\mid _N}\in$ End$_R(N)$. This completes the proof.
\end{proof}

 Let $M$ be a module and $N$ a submodule of $M$. In the
next result, we determine under what conditions the Mitsch order
on $M$ can be described on $M/N$.

\begin{prop}
 Let $M$ be a module and $N$ a submodule of $M$. If $N$ is fully invariant,
  then Mitsch order on $M$ can be defined on $M/N$.
\end{prop}

\begin{proof}
Let $M$ be a module, $N$ a fully invariant submodule of $M$ and
$m_1, m_2\in M$. Assume that $m_1\leq_{M} m_2$ on $M$. Then there
exists $f\in S$ and $a\in R$ such that $m_1 = m_2a=fm_2$ and $m_1
= fm_1$. Define $g\in$ End$_R(M/N)$ as $g(m+N)=fm+N$ for $m+N\in
M/N$. Hence $g$ is well defined map on $M/N$. In fact, $m + N = m_1
+ N$ implies $m - m_1\in N$. Since $N$ fully invariant, $f(m -
m_1) = f(m) - f(m_1) \in f(N)\subseteq N$. Thus $g(m + N) = f(m)
+ N = f(m_1) + N$. Since $f$ is an endomorphism of $M$, $g$ is
also an endomorphism of $M/N$. Let $m_1+N$,  $m_2+N \in M/N$ with
$m_1\leq_{M} m_2$ on $M$. By assumption $m_1 = m_2a=fm_2$ and $m_1
= fm_1$. Then $m_1 + N = m_2a + N = (m_2 + N)a$ and $m_1 + N =
fm_2 + N = g(m_2 + N)$. $m_1 + N = fm_1 + N = g(m_1 + N)$. It
entails that $m_1 + N\leq_{M} m_2 + N$ on $M/N$.
\end{proof}

\begin{cor}
Let $R$ be a ring and $I$ an ideal of $R$. Then the Mitsch order
on $R$ can be defined on $R/I$.
\end{cor}

\section{Orders Versus The Natural Partial Order On Modules}

In this section, we investigate the relations between the Mitsch
order and the other well-known orders, namely, the minus order,
the Jones order, the direct sum order and the space pre-order on
modules.

Let $M$ be a module and $m_1, m_2\in M$. In
\cite{UHHM-comm}, recall that $m_1\leq^{-} m_2$ if there exists $\varphi \in
M^*$ such that $m_1=m_1\varphi m_1$, $m_1\varphi=m_2 \varphi$, and
$\varphi m_1=\varphi m_2$. The relation $\leq^{-}$ is called the
\emph{minus order} on $M$.

\begin{prop} \label{minussamitsch}
Let $M$ be a module and $m_1, m_2\in M$. If $m_1\leq^{-} m_2$, then $m_1\leq_{M}
m_2$. The converse holds if $M$ is a regular module.
\end{prop}

\begin{proof}
If $m_1\leq^{-} m_2$, then there exists $\varphi \in
M^*$ such that $m_1=m_1\varphi m_1$, $m_1\varphi=m_2 \varphi$, and
$\varphi m_1=\varphi m_2$. Take $a=\varphi m_1\in R$ and $f=m_1\varphi \in S$. Then $m_1=m_1\varphi m_1=m_2\varphi m_1=m_2a$, $m_1=m_1\varphi m_1=m_1\varphi m_2=fm_2$ and $m_1=m_1\varphi m_1=fm_1$. Therefore $m_1\leq_{M} m_2$. Conversely, if $M$ is a regular module and $m_1\leq_{M}m_2$, then from \cite[Theorem 5.6]{UHHM-comm} we have $m_1\leq^{-} m_2$.
\end{proof}

The following example shows that the regularity
condition in Proposition \ref{minussamitsch} is not superfluous.

\begin{ex} \rm \label{mitschminuscounterex}
Consider $M=\mathbb{Z}_{12}$ as a $\mathbb{Z}$-module. Since
$M^*=0$, $M$ is not a regular module. Let $a=3\in \mathbb{Z}$ and
the homomorphism $f\in S$ defined by $f(\bar{n})=\bar{n}a$ for all
$\bar{n}\in M$ we have $\bar{6}=\bar{2}a=f(\bar{2})$ and
$\bar{6}=f(\bar{6})$ which yields that $\bar{6}\leq_{M} \bar{2}$.
But since $M^*=0$, $\bar{6}\nleq^{-} \bar{2}$.
\end{ex}

In \cite{UHHM-comm}, the Jones order on modules is defined  in the
following way: Let $M$ be a module and $m_1,m_2\in M$. Then we
write $m_1\leq _{J}m_2$ if there exist
 $f^2=f \in S$ and $a^2=a\in R$ such that $m_1 =m_2a=fm_2$. In the
 next result, we obtain the relation between the Mitsch order and
 the Jones order.

\begin{prop} \label{jonessamitsch}
 Let $M$ be a module and $m_1, m_2\in M$. If $m_1\leq _{J}m_2$, then $m_1\leq_{M} m_2$.
\end{prop}

\begin{proof}
Asume that $m_1\leq _{J}m_2$. Then $m_1 =m_2a=fm_2$ for some $f^2=f \in S$ and $a^2=a\in R$. From $m_1 =fm_2$ we obtain $fm_1 =fm_2=m_1$. Therefore $m_1\leq_{M} m_2$.
\end{proof}

The following example shows that there are Mitsch ordered elements
which are not Jones ordered in a module.

\begin{ex}\rm \label{MitJon} Let $C$ be the commutative semigroup on generators $a$, $b$ and
$c$ with relations $ac = a = cb$ as in \cite[Example 2.9]{Hi}.
Consider the semigroup ring $R = \mathbb Z_2[C]$. Being $ac = a =
cb$
 implies $ab = acb = a^2$. It can be listed the distinct elements of $C$ and $R$ as
$$C = \{a, a^2,\dots, a^n~,\dots;b, b^2,\dots,b^n,\dots;c, c^2,\dots,c^n
,\dots\}$$ $$R = \{\sum_{i}\sum_{j}\sum_{k} a^ib^jc^k\mid a, b,
c\in C\}.$$ It is easy to see that $C$ is idempotent-free and the
idempotents of $R$ are trivial. The commutativity of $C$ and $R$
entails that in the equalities $m_1 = m_2d =fm_2,$ $m_1 = fm_1$ we
replace $m_1$ by $a$, $m_2$ by $b$, $d$ by $c$ and $f$ by $c$ to
get $m_1 = m_2d$ to $a = bc$, $m_1 = fm_2$ to $a = cb$ and $m_1 =
fm_1$ to $a = ca$. Then $m_1 = m_2d = fm_2$,  $m_1 = fm_1$ is $a =
bc = cb,$ $a = ca$. Hence we have $a\leq_M b$. Since $R$ has only
trivial idempotents $\{0, 1\}$, $a\nleq_J b$.
\end{ex}

A natural question is to ask under which conditions Mitsch ordered
elements of a module are Jones ordered.
\begin{thm} \label{mitschsejones} Let $M$ be a module and $m_1, m_2\in M$. Then we have the following.
\begin{enumerate}
\item If $m_1\leq_{M} m_2$ and $r_{R}(m_2)=l_{S}(m_2)=0$, then $m_1\leq _{J}m_2$.
\item If $m_1\leq_{M} m_2$ and $m_1$ is regular, then $m_1\leq _{J}m_2$.
\end{enumerate}
\end{thm}
\begin{proof} (1) Let $m_1\leq_{M} m_2$, then there exist $f \in S$ and $a \in R$ such that $m_1=m_2a=fm_2$ and $m_1=fm_1=m_1a$. Hence $fm_2=fm_1=f^2m_2$ and $m_2a=m_1a=m_1a^2$. Thus $(f-f^2)m_2=0$ and $m_2(a-a^2)=0$. Since $r_{R}(m_2)=l_{S}(m_2)=0$ then $f\in S$ and $a\in R$ are idempotent elements. Therefore $m_1\leq _{J}m_2$.\\
(2) Let $m_1\in M$ be regular and $m_1\leq_{M} m_2$. The regularity of $m_1$ yields there exists $\varphi\in M^*$ such that $m_1 = m_1\varphi m_1$. Then $m_1\varphi\in S$ is an idempotent and $\varphi m_1\in R$ is an idempotent.  Consider $m_1 = m_1\varphi m_1 = (m_2r)(\varphi m_1) = m_2(r\varphi m_1) = m_2e$ where $e = r\varphi m_1$. We claim that $e^2 = e$. In fact $e^2 = (r\varphi m_1)(r\varphi m_1) = (r\varphi)(m_1r)(\varphi m_1) = (r\varphi)(m_1)(\varphi m_1) = r\varphi m_1 = e$. Next we consider $m_1 = m_1\varphi m_1 = (m_1\varphi)(sm_2) = (m_1\varphi s)m_2 = fm_2$ where $m_1\varphi s = f\in S$ and we claim $f^2 = f$. In fact $f^2 = (m_1\varphi s)(m_1\varphi s) = (m_1\varphi) (sm_1)(\varphi s) = (m_1\varphi) (m_1)(\varphi s) = m_1\varphi s = f$. The net result is $m_1 = m_2e$ and $e^2 = e\in R$ and $m_1 = fm_2$ and $f^2 = f\in S$. It follows that $m_1\leq _J m_2$, $m_1 = fm_2 = m_2e =fm_1e = fm_2e$.
\end{proof}

Recall that the direct sum order on modules is defined in
\cite{UHHM-comm} as follows: Let $M$ be a module and $m_1,m_2\in
M$. Then $m_1\leq ^{\oplus}m_2$ if $m_2R=m_1R \oplus (m_2-m_1)R$.

\begin{prop} \label{mitschseoplus}
Let $M$ be a module and $m_1, m_2\in M$. If $m_1\leq _{M}m_2$, then $m_1\leq^{\oplus} m_2$. The converse holds if $M$ is a regular module.
\end{prop}

\begin{proof}
If $m_1\leq _{M}m_2$, then $m_1R\subseteq m_2R$ and $(m_2-m_1)R\subseteq m_2R$ from Proposition \ref{mitschann}(3). It is clear that $m_2R=m_1R + (m_2-m_1)R$. Let $m\in m_1R\cap(m_2-m_1)R$ then there exists $a,b\in R$ such that $m=m_1a=(m_2-m_1)b$. Since $m_1\leq _{M}m_2$, $m_1=m_1r=m_2r$ for some $r\in R$. Thus $m=m_2ra=(m_2-m_2r)b$. Hence $m_2(ra+rb-b)=0$. From Proposition \ref{mitschann}(1) $m_1(ra+rb-b)=0=m_1a+m_1b-m_1b=m_1a=m$. Therefore $m_1R\cap(m_2-m_1)R=0$ and $m_2R=m_1R \oplus (m_2-m_1)R$. The converse is clear from \cite[Theorem 5.11]{UHHM-comm} and Proposition \ref{minussamitsch}.
\end{proof}

As the following example shows, the condition being regular is not superfluous in Proposition \ref{mitschseoplus}.

\begin{ex}\rm
Let $F$ be a field and $R = \left \{\begin{bmatrix}a&0&b\\0&c&d\\0&0&e\end{bmatrix}\mid a,b,c,d,e\in F\right \}$ be the subring of $U_3(F)$ the ring of $3\times 3$ upper triangular matrices over $F$.\\
Let $m_1 = \begin{bmatrix}0&0&1\\0&0&1\\0&0&0\end{bmatrix}$ and $m_2 = \begin{bmatrix}0&0&1\\0&1&1\\0&0&0\end{bmatrix}$, and so $m_2 - m_1 =  \begin{bmatrix}0&0&0\\0&1&0\\0&0&0\end{bmatrix}$. Then $m_1R = \left\{\begin{bmatrix}0&0&a\\0&0&a\\0&0&0\end{bmatrix}\mid a\in F\right\}$, $m_2R = \left\{\begin{bmatrix}0&0&b\\0&c&d\\0&0&0\end{bmatrix}\mid b,c,d\in F\right\}$ and \linebreak $(m_2 - m_1)R = \left\{\begin{bmatrix}0&0&0\\0&s&t\\0&0&0\end{bmatrix}\mid s,t\in F\right\}$.
Since $F$ is a field, it is easy to check that $m_2R = m_1R\oplus (m_2 - m_1)R$. So $m_1\leq^{\oplus}m_2$. On the other hand, we claim that there is not an element $f$ in $R$ such that $m_1 = fm_2$. Otherwise, let $f = \begin{bmatrix}a&0&b\\0&z&u\\0&0&e\end{bmatrix}\in R$ such that $\begin{bmatrix}0&0&1\\0&0&1\\0&0&0\end{bmatrix} = \begin{bmatrix}a&0&b\\0&z&u\\0&0&e\end{bmatrix}\begin{bmatrix}0&0&1\\0&1&1\\0&0&0\end{bmatrix} = \begin{bmatrix}0&0&a\\0&z&z\\0&0&0\end{bmatrix}$. It entails that $1 = 0$. This contradiction shows that $m_1\nleq_M m_2$.
\end{ex}

In \cite{MITRA}, the space pre-order on complex matrices was
introduced by Mitra. Recently, the space pre-order is extended to
modules in \cite{UHHM-deb}. Let $M$ be a module and $m_1,m_2\in
M$. Then $m_1\leq _{S}m_2$ if $Sm_1\subseteq Sm_2$ and
$m_1R\subseteq m_2R$.

\begin{prop} \label{mitschsespace}
 Let $M$ be a module and $m_1, m_2\in M$. If $m_1\leq_{M} m_2$, then $m_1\leq _{S}m_2$.
\end{prop}

\begin{proof}
It is clear from Proposition \ref{mitschann}.
\end{proof}

One may suspect that the converse of Proposition
\ref{mitschsespace} holds. But the following example erases the
possibility.

\begin{ex} \rm
Let $R=\begin{bmatrix}
       \mathbb{Z} & 2\mathbb{Z}\\
       2\mathbb{Z} & \mathbb{Z}
       \end{bmatrix}$, and $M=R_R$. Consider
       $m_1=\begin{bmatrix}
       9 & 6 \\
       6 & 3
       \end{bmatrix}$,
       $m_2=\begin{bmatrix}
       1 & 2 \\
       2 & 1
       \end{bmatrix}$,
       $a=\begin{bmatrix}
       1 & 0 \\
       4 & 3
       \end{bmatrix}$ and
       $f=\begin{bmatrix}
       1 & 4 \\
       0 & 3
       \end{bmatrix}$
       $\in M$. Since $m_1=m_2a=fm_2$ we have $m_1\leq_{S} m_2$. It can be seen that $m_1\nleq_{M} m_2$.
\end{ex}


We close this section by obtaining the relation between the minus order and the Jones order. So the following result completes the diagram which is given at the end of Introduction. 

\begin{prop} \label{minussajones}
Let $M$ be a module and $m_1, m_2\in M$. If $m_1\leq^{-} m_2$, then $m_1\leq _{J}m_2$.
\end{prop}

\begin{proof}
 Let $M$ be a module and $m_1, m_2\in M$. If $m_1\leq^{-} m_2$,
 then there exist $f^2=f\in S, a^2=a\in R$ such that
$m_1=fm_1=fm_2=m_1a=m_2a$, by \cite[Lemma 5.2]{UHHM-comm}.
\end{proof}

The next example shows that the converse of Proposition \ref{minussajones} need not be true.
\begin{ex}\em
Let $n$ be a positive integer and consider the $\mathbb Z$-module $\mathbb Z_n$. Then the dual module $\mathbb Z_n^*=0$. Let $0\neq m\in \mathbb Z_n$. 
On the one hand, since $\mathbb Z_n^*=0$, we have $m\nleq^- m$. On the other hand, clearly, $m\leq_J m$.
\end{ex}


\section{Lattice Properties of the Mitsch Order}

In this section, we study lattice properties of the Mitsch order
on modules. Hence we present various characterizations of the
Mitsch order as an application. We characterize the maximal
elements of modules with respect to the Mitsch order.
We investigate the compatibility of the  Mitsch order  with the addition and multiplication. We give examples to show that the Mitsch order is not compatible with the addition and multiplication in general.

Let $R$ and $S$ be rings, $_RV_S$ and $_SW_R$ be bimodules. Then
the 4-tuple $(R,V,W,S)$ is said to be a \emph{Morita context} if
there exist products $V \times W \rightarrow R$ defined by
$(v,w)\mapsto vw$ and $W \times V \rightarrow S$ defined by
$(w,v)\mapsto wv$. With these requirements, $\begin{bmatrix}R&
V\\W&S\end{bmatrix}$ is an associative ring using matrix
operations and it is called a \emph{Morita context ring}
\cite{morita}.

\begin{rem}{\rm Let $M$ be a module. Then $_SM_R$ is a bimodule.
Also, $M^*$ is a right $S$-module via composition of maps, and
$M^*$ becomes a left $R$-module as follows: For $\varphi \in M^*$
and $r\in R$ define $r \varphi \in M^*$ by $(r \varphi)x=r(\varphi
x)$ for all $x\in M$. Hence we have $_SM_R$ and $_R(M^*)_S$
bimodules. Also, if $m\in M$ and $\varphi \in M^*$, then $\varphi
m \in R$ and dually $m\varphi \in S$ with $(m\varphi)x=m(\varphi
x)\equiv m\varphi x$ for all $x\in M$. If $m\in M$ is regular,
then $m=m\varphi m$ for some $\varphi\in M^*$, so we have
idempotent elements $\varphi m \in R$ and $m \varphi \in S$. It is
easy to verify that $(R,M^*,M,S)$ is a Morita context, called
\textit{the standard context} of the module $M$.}\end{rem}

In \cite{Unit}, Chen et al. defined \textit{left invertible} and
\textit{right invertible} elements of a module by using the Morita
context as follows. Let $(R,M^*,M,S)$ be a standard context of the
module $M$. Then
\begin{equation*}
  m\in M~~\text{is called}~~\begin{cases}
    {\textit left ~invertible}~ \text{if} ~\varphi m=1_R ~\text{for some}~ \varphi\in M^{*},\\
    {\textit right ~ invertible}~ \text{if}~ m\varphi=1_S~ \text{for some}~ \varphi\in M^{*}.
  \end{cases}
\end{equation*}

\begin{thm} \label{maxicin} Let $M$ be a module and $m_1, m_2\in M$. If $m_1\leq_{M} m_2$ then the following hold.
\begin{enumerate}
    \item If $m_1$ is left invertible, then $m_1=m_2$.
    \item If $m_1$ is right invertible, then $m_1=m_2$.
    \item If $m_2\in m_1R$, then $m_1=m_2$.
    \item If $m_2\in Sm_1$, then $m_1=m_2$.
    \item If $m_2\in N$ for a submodule $N$ of $M$, then $m_1\in N$.
    \item $m_2-m_1\leq_{M} m_2$.
\end{enumerate}
\end{thm}

\begin{proof}
(1) Let $m_1\leq_{M} m_2$ and $m_1$ is left invertible.
Then there exist $a\in R$ and $f\in M^*$ such that $m_1=m_1a=m_2a$ and $fm_1=1_R$.
Hence $a=fm_1a=fm_1=1_R$. Therefore $m_1=m_2$.\\
(2) It can be proved similar to the proof of (1).\\
(3) Let $m_1\leq_{M} m_2$ and $m_2\in m_1R$.
 Then there exist $b\in R$ and $f\in S$ such that $m_2=m_1b$ and $fm_1=fm_2=m_1$.
 Thus $m_2=m_1b=fm_1b=fm_2=m_1$.\\
(4) It can be proved similar to the proof of (3).\\
(5) It is clear.\\
(6) Let $m_1\leq_{M} m_2$. Then $m_1 =m_2a=fm_2$ and
$m_1 =fm_1$ for some $f \in S$ and $a\in R$.
Therefore $m_2-m_1=m_2(1-a)=(1-f)m_2$ and $m_2-m_1 =(1-f)(m_2-m_1)$.
Hence $m_2-m_1\leq_{M} m_2$.
\end{proof}

Note that zero is the minimum element of the poset $(M,\leq_{M})$.
In the following, we determine the maximal elements of the module
with respect to the Mitsch order.

\begin{thm} \label{maxelement}
Let $M$ be a module and $m\in M$. Then $m$ is a maximal element of the the poset $(M,\leq_{M})$ if any of the following holds.
\begin{enumerate}
    \item m is left invertible.
    \item m is right invertible.
    \item For every $b\in r_{R}(m)$, $b+1$ is invertible.
    \item For every $g\in l_{S}(m)$, $g+1$ is invertible.
    \item $r_{R}(m)=0$.
    \item $l_{S}(m)=0$.
\end{enumerate}
\end{thm}

\begin{proof}
(1) and (2) follow from Theorem \ref{maxicin}.\\
(3) Let $n\in M$ and $m\leq_{M} n$, then there exists $a \in R$ such that $m=na=ma$.
Hence $m(a-1)=0$.
Therefore $a\in R$ is invertible and $na=ma$ implies that $n=m$.
Thus m is a maximal element of the poset $(M,\leq_{M})$.\\
(4) It can be proved similar to the proof of (3).\\
(5) and (6) are clear from (3) and (4), respectively.
\end{proof}

 The following corollary is a direct consequence of
Theorem \ref{maxicin} and Theorem \ref{maxelement}.

\begin{cor} Let $M$ be module.
\begin{enumerate}
    \item If $M_R$ or ${}_SM$ is a faithful module, then every nonzero element of $M$ is maximal with respect to $\leq_{M}$.
    \item If $M_R$ or ${}_SM$ is a cyclic module generated by m, then $m$ is a maximal element of $M$ with respect to $\leq_{M}$.
    \item All torsion-free elements of $M$ are maximal with respect to $\leq_{M}$.
\end{enumerate}
\end{cor}

One may suspect that a module $M$ becomes a lattice with respect
to the Mitsch order. But the following remark reveals that this is not the case.

\begin{rem}\rm Let $M$ be a regular module and $m_1, m_2\in M$.
By Proposition \ref{minussamitsch},  $m_1\leq_{M} m_2$ implies $m_1\leq^{-} m_2$. Assume that $m_1$ and $m_2$ have an upper bound $m\in M$ in the lattice of the Mitsch order on $M$. Then $m$ is also an upper bound of $m_1$ and $m_2$ in the lattice of minus order on $M$. It is known by \cite[Example 3.6]{ACHHU} that the minus order on a regular module $M$ need not form a lattice on $M$. Therefore $M$ need not be a lattice with respect
to the Mitsch order in general.
\end{rem}

In the following example, we show that the Mitsch order is not
always compatible with the scalar multiplication.

\begin{ex} \label{compatiblemultiplication} \rm
Let $R=\begin{bmatrix}
       \mathbb{Z} & \mathbb{Z}_6 \\
       0 & \mathbb{Z}_6
       \end{bmatrix}$ and $M=R_R$. Consider
       $m_1=\begin{bmatrix}
       2 & \bar{2} \\
       0 & \bar{3}
       \end{bmatrix}$ and
       $m_2=\begin{bmatrix}
       2 & \bar{0} \\
       0 & \bar{1}
       \end{bmatrix}$ in $M$. Then for
       $a=\begin{bmatrix}
       1 & \bar{1} \\
       0 & \bar{3}
       \end{bmatrix}$,
       $f=\begin{bmatrix}
       1 & \bar{2} \\
       0 & \bar{3}
       \end{bmatrix} \in R $ we have $m_1=m_2a=fm_2$ and $m_1=fm_1$.
        Hence $m_1\leq_{M} m_2$. Set
       $g=\begin{bmatrix}
       2 & \bar{0} \\
       0 & \bar{3}
       \end{bmatrix} \in R $. Since there is no solution for the equation $\overline{3x}=4$
        in $\mathbb{Z}_6$, we obtain
       $gm_1\nleq_{M} gm_2$.

\end{ex}

We have the following result for which the Mitsch order is
compatible with the scalar multiplication. Recall that a ring $R$ is called a {\it right (resp., left) duo ring} if every right (resp., left) ideal of $R$ is an ideal; equivalently $Ra\subseteq aR~(\mbox{resp.,}~ aR\subseteq Ra)$ for every $a\in R$.

\begin{prop} \label{mitschinvertible} Let $M$ be a module and $m_1, m_2\in M$.  Then we have the following.
\begin{enumerate}
    \item $m_1\leq_{M} m_2$ if and only if $gm_1\leq_{M} gm_2$
    for every invertible element $g \in S$.
     \item $m_1\leq_{M} m_2$ if and only if $m_1b\leq_{M} m_2b$
    for every invertible element $b \in R$.
    \item If $m_1\leq_{M} m_2$ and $b\in C(R)$, then $m_1b\leq_{M} m_2b$.
    \item If $m_1\leq_{M} m_2$ and $g\in C(S)$, then $gm_1\leq_{M} gm_2$.
    \item If $m_1\leq_{M} m_2$ and $R$ is a right duo ring, then $m_1b\leq_{M} m_2b$ for all $b\in R$.
    \item If $m_1\leq_{M} m_2$ and $S$ is a left duo ring, then $gm_1\leq_{M} gm_2$ for all $g\in S$.

\end{enumerate}\end{prop}

\begin{proof}
(1) Assume that $m_1\leq_{M} m_2$. Then there exist $f \in S$ and $a \in R$ such that $m_1=m_2a=fm_2$ and $m_1=fm_1$. Let $g \in S$ be invertible. Then $gm_1=gm_2a=gfm_2$ and $gm_1=gfm_1$. Since $g$ is invertible, $gm_1 =gm_2a=gfg^{-1}gm_2$. Therefore $gm_1\leq_{M} gm_2$. Conversely, if $gm_1\leq_{M} gm_2$ for an invertible $g\in S$, then $g^{-1}gm_1\leq_{M} g^{-1}gm_2$ and so $m_1\leq_{M} m_2$.\\
(2) It can be proved similar to the proof of (1).\\
(3) Assume that $m_1\leq_{M} m_2$. Then there exist $f \in S$ and $a \in R$ such that $m_1=m_2a=fm_2$ and $m_1=fm_1$. Since $b \in C(R)$, $m_1b=m_2ab=m_2ba=fm_2b$ and $m_1b=fm_1b$. Therefore $m_1b\leq_{M} m_2b$.\\
(4) It can be proved similar to the proof of (3).\\
(5) Let $m_1\leq_{M} m_2$ and $b\in R$. Then $m_1=fm_2=m_2a$ and $m_1=m_1a$ for some $f \in S$ and $a \in R$. It follows that $m_1b=fm_2b=m_2ab=m_1ab$. Since $R$ is right duo, there exists $c\in R$ such that $ab=bc$. Thus $m_1b\leq_{M} m_2b$.\\
(6) It can be proved similar to the proof of (5).
\end{proof}

In Example \ref{compatiblemultiplication}, we show that the Mitsch
order is not compatible with the scalar multiplication. Following
example shows that the Mitsch order also is not compatible with
addition.

\begin{ex} \rm
Let $R=M_2(\mathbb{Q})$ and $M=R_R$. Consider
       $m_1=\begin{bmatrix}
       2 & 5 \\
       0 & 0
       \end{bmatrix}$,
       $m_2=\begin{bmatrix}
       1 & 5 \\
       1 & 0
       \end{bmatrix}$,
       $f=\begin{bmatrix}
       1 & 1 \\
       0 & 0
       \end{bmatrix}$ and
       $a=\begin{bmatrix}
       0 & 0 \\
       \frac{2}{5} & 1
       \end{bmatrix}$
       $\in M$. Since $m_1=m_2a=fm_2$ and $m_1=fm_1$ we have $m_1\leq_{M} m_2$ and $0\leq_{M} m_2$. Since there is no $b\in R$ such that $(m_2+m_2)b=m_1b=m_1$ we have
       $m_1\nleq_{M} m_2+m_2$.
\end{ex}

We give the following result in which cases the Mitsch order is
compatible with the addition.

\begin{prop}
Let $M$ be a module and $m_1, m_2\in M$. Then the following statements are equivalent.
\begin{enumerate}
    \item $m_1\leq_{M} m_2$.
    \item $nm_1\leq_{M} m_2$ for all $n\in \mathbb{Z}$.
    \item $nm_1\leq_{M} nm_2$ for all $n\in \mathbb{Z}$.
\end{enumerate}
\end{prop}

\begin{proof}
It is clear.
\end{proof}

\begin{prop}\label{epi-duo}
Let $M$ be a module and $m\in M$. 
\begin{enumerate}
    \item If $mR$ is a fully invariant submodule of $M$, then $fm\leq_{M} m$ for every $f^2=f\in S$.
    \item If $mR$ is a homomorphic image of $M$, then $me\leq_{M} m$ for every $e^2=e\in R$.
\end{enumerate}
\end{prop}

\begin{proof} (1)
Assume that $f^2=f\in S$ and $mR$ is fully invariant in $M$. Then $fm=ffm$ and since $fm\in mR$, there exists $a\in R$ such that $fm=ma$. Therefore $fm\leq_{M} m$.\\
(2) Similar to the proof of (1).
\end{proof}

The next result is a direct consequence of Proposition \ref{epi-duo}.
\begin{cor}
Let $M$ be a module. 
\begin{enumerate}
    \item If $M$ is duo, then $fm\leq_{M} m$ for all $m\in M$ and $f^2=f\in S$.
    \item If $M$ is 1-epi-retractable, then $me\leq_{M} m$ for all $m\in M$ and $e^2=e\in R$.
\end{enumerate}
\end{cor}


\begin{thebibliography}{99}

\bibitem{ACHHU} S. R. Argun, T. P. Calci, S. Halicioglu, A. Harmanci and B. Ungor, \textit{Lattice properties of the minus order on regular modules}, Rocky Mountain J. Math., 52(1)(2022), 15–27.




\bibitem{Unit} H. Chen, W. K. Nicholson and Y. Zhou, \textit{Unit-regular
modules}, Glasg. Math. J., 60(1)(2018), 1-15.

\bibitem {DjordjevicRakicMarovt} D. S. Djordjevi\'{c}, D. S. Raki\'{c} and J.
Marovt, \textit{Minus partial order in Rickart rings}, Publ. Math.
Debrecen, 87(3-4)(2015), 291-305.

\bibitem {DKM} G. Dolinar, B. Kuzma and J. Marovt \textit{A note on partial orders of Hartwig, Mitsch, and Šemrl}, Appl. Math. Comput., 270 (2015), 711–713.


\bibitem {Drazin} M. P. Drazin, \textit{Generalizations of Fitting's lemma in arbitrary associative rings},
Comm. Algebra, 29(8)(2001), 3647-3675.


\bibitem {Hartwig}R. E. Hartwig, \textit{How to partially order regular elements},
Math. Japon., 25(1)(1980), 1-13.

\bibitem{HS} R. E. Hartwig and G. P. H. Styan, \textit{On some characterizations of the “star” partial ordering for matrices and rank substractivity},  Linear Algebra Appl., 82(1986),
145-161.

\bibitem {H}A. Hattori, \textit{A foundation of torsion theory for modules over general rings},
Nagoya Math. J., 17(1960), 147-158.



\bibitem{Hi} P. M. Higgins, \textit{The Mitsch order on a semigroup}, Semigroup Forum, 49(2)(1994), 261-266.



\bibitem {M}J. Marovt, \textit{On partial orders in Rickart rings}, Linear Multilinear
Algebra, 63(9)(2015), 1707-1723.

\bibitem{MITRA} S. K. Mitra, \textit{Matrix partial order through generalized inverses: unified theory}, Linear Algebra Appl., 148(1991), 237-263.

\bibitem {Mitsch} H. Mitsch, \textit{A natural partial order for semigroups}, Proc.
Amer. Math. Soc., 97(3)(1986), 384-388.

\bibitem {morita} K. Morita, \textit{Duality for modules and its applications to the theory
of rings with minimum condition}, Sci. Rep. Tokyo Kyoiku Daigaku
Sect. A, 6(1958), 83-142.









\bibitem{UHHM} B. Ungor, S. Halicioglu, A. Harmanci and J. Marovt,
\textit{Partial orders on the power sets of Baer rings}, J. Algebra Appl., 19(2020), 2050011 (14 pp).

\bibitem{UHHM-comm} B. Ungor, S. Halicioglu, A. Harmanci and J. Marovt,
\textit{Minus partial order in regular modules}, Comm. Algebra,
48(10)(2020), 4542-4553.

\bibitem{UHHM-deb} B. Ungor, S. Halicioglu, A. Harmanci and J. Marovt,
\textit{On properties of the minus partial order in regular modules}, Publ. Math. Debrecen, 96(1-2)(2020), 149-159.

\bibitem{ZEL}  J. Zelmanowitz, \textit{Regular modules}, Trans. Amer. Math. Soc., 163(1972), 341-355.




\end{thebibliography}
\end{document}